\documentclass[a4paper,12pt]{amsart}
\usepackage{geometry}                		% See geometry.pdf to learn the layout options. There are lots.
\usepackage{graphicx}				% Use pdf, png, jpg, or eps§ with pdflatex; use eps in DVI mode
\usepackage{amssymb}

\usepackage{latexsym,exscale,enumerate,amsfonts,amssymb,mathtools}
\usepackage{amsmath,amsthm,amsfonts,amssymb,amscd,stmaryrd,textcomp}
\usepackage[normalem]{ulem}
\usepackage{young}
\usepackage{thmtools,xcolor}
\usepackage{easybmat}

\definecolor{colormy}{rgb}{0.8,0.05,0.05}
\definecolor{mycolor}{rgb}{0.25,0.99,0.25}

\usepackage[all]{xy}
\SelectTips{cm}{}

\usepackage{tikz}
\usetikzlibrary{decorations.markings}
\usetikzlibrary{decorations.pathreplacing}
\usetikzlibrary{arrows,shapes,positioning}
\tikzstyle directed=[postaction={decorate,decoration={markings,
    mark=at position #1 with {\arrow{>}}}}]
\tikzstyle rdirected=[postaction={decorate,decoration={markings,
    mark=at position #1 with {\arrow{<}}}}]

\usepackage{hyperref}

\newcommand{\Hom}{\mathrm{Hom}}

\newcommand{\Ext}{\mathrm{Ext}}

\newcommand{\h}{\mathrm{ht}}

\def\Z{{\mathbb Z}}

\def\Ind{\mathrm{Ind}}

\theoremstyle{definition}
\newtheorem{thm}{Theorem}[section]
\newtheorem{cor}[thm]{Corollary}
\newtheorem{lem}[thm]{Lemma}
\newtheorem{prop}[thm]{Proposition}

\theoremstyle{definition}
\newtheorem{examplecounter}{Example}

\numberwithin{equation}{section}

\declaretheorem[style=definition,name=Definition,qed=$\blacktriangle$,numberlike=thm]{defn}
\declaretheorem[style=definition,name=Remark,qed=$\blacktriangle$,numberlike=thm]{rem}

\title{$p$-filtrations of dual Weyl modules}

\author{Henning Haahr Andersen}

\address{Three Gorges Mathematical Research Center, China Three Gorges University, Yichang, 443002  Hubei , China}
\email{h.haahr.andersen@gmail.com}

\date{}							% Activate to display a given date or no date

\begin{document}

\begin{abstract}
Let $G$ be a semisimple algebraic group over a field of characteristic $p > 0$. We prove that the dual Weyl modules for $G$ all have $p$-filtrations when $p$ is not too small. Moreover, we give applications of this theorem to $p^n$-filtrations for $n > 1$, to modules containing the Steinberg module as a tensor factor, and to the Donkin conjecture on modules having $p$-filtrations.

\end{abstract}

\maketitle

\section{Introduction}

Let $k$ be an algebraically closed field of characteristic $p>0$ and denote by $G$ a connected semisimple algebraic group over $k$. Pick a maximal torus $T$ in $G$ and a Borel subgroup $B$ containing $T$. We let $X = X(T)$ denote the character group of $T$. Then in the root system $R \subset X$ for $(G, T)$  we choose $R^+$ as the subset of $R$ with $-R^+$ equal to the roots of $B$. The corresponding positive chamber $X^+ \subset X$ we call the set of dominant weights. We denote by $S \subset R^+$ the set of simple roots. If $\alpha$ is any root in $R$ we write $\alpha^\vee$ for the corresponding coroot. 

When $\lambda \in X$ the corresponding dual Weyl module is
$$\nabla(\lambda) = \Ind_B^G \lambda.$$
Here $\lambda$ is considered as a character of $B$ obtained by extending $\lambda : T \rightarrow k^*$ to $B$ by letting it be trivial on the unipotent radical of $B$. Recall that then $\nabla (\lambda) = 0$ unless $\lambda \in X^+$. 

For $\lambda \in X^+$ the module $\nabla(\lambda )$ has a unique simple submodule which we denote $L(\lambda)$. These simple modules constitute up to isomorphisms a complete list of finite dimensional simple $G$-modules. The subset $\{L(\lambda) \mid \lambda \in X_1 \}$ where $X_1 = \{\lambda \in X \mid 0 \leq \langle \lambda, \alpha^\vee \rangle < p \text { for all } \alpha \in S$\} is then called the set of restricted simple modules, and $X_1$ is the set of restricted weights.  

Recall that we have a Frobenius homomorphism $F: G \rightarrow G$. We shall assume that $F$ is chosen as in \cite[Section II.3.1]{RAG}. The kernel of $F$ is denoted $G_1$. This is an infinitesimal normal subgroup scheme of $G$. For any closed subgroup $H \subset G$ we then also have the group scheme $G_1 H$. In particular, we shall need the group scheme $G_1B$ and more generally $G_1P$ where $P$ is a parabolic subgroup containing $B$. 

If $M$ is a $G$-module the Frobenius twist of $M$ is denoted $M^{(1)}$. As a vector space $M^{(1)} = M$ but the action of $G$ is twisted by $F$: if $g \in G, m \in M^{(1)}$ then $g  m = F(g) m$.

A $p$-filtration of a finite dimensional $G$-module $M$ is a filtration with quotients of the form $L(\lambda) \otimes \nabla(\mu)^{(1)}$ where $\lambda \in X_1$ and $\mu \in X^+$. The main aim of this note is to prove that dual Weyl modules have $p$-filtrations. Our main result is

\begin{thm} \label{intro}
Suppose $p \geq (h-2)h$. Then for any $\lambda \in X^+$ the dual Weyl module $\nabla(\lambda)$ has a $p$-filtration.
\end{thm}

Our method for proving this result is fairly simple: We exploit the fact that induction from $B$ to $G$ may be done in two or three steps. First we induce $\lambda \in X^+$ from $B$ to $G_1B$. We then take a composition series for this $G_1B$-module and induce it from $G_1B$ to $G$. It turns out that this yields a $p$-filtration of $\nabla(\lambda)$ (at least when $p$ satisfies the given bound). We check this by breaking $\Ind_{G_1B}^G$ up as $\Ind_{G_1P}^G \circ \Ind_{G_1B}^{G_1P}$, where $P$ is a certain parabolic subgroup depending on $\lambda$.

For a given $\lambda \in X^+$ we give bounds on $p$ which in many cases are weaker than the general bound in this theorem. For instance, if $\lambda$ is not close to the walls of $X^+$ then $\nabla(\lambda)$ has a $p$-filtration for all primes $p$, see Theorem \ref{small and large}(2)  below.  If $\lambda$ is close to just one wall, the bound can be relaxed to $p \geq 2(h-2)$. 

Once we have established this we turn (in Section 3) to
 $p^n$-filtrations for arbitrary $n \geq 0$ (defined completely analogously to the $n = 1$ case). We prove that if $M$ has a $p^n$-filtration for some $n \geq 0$ then our result implies that $M$ also has a $p^r$-filtration for all $r \geq n$. In particular, all dual Weyl modules have $p^n$ filtrations for all $n$ (and all $p$ satisfying the assumptions in Theorem \ref{intro}). We also prove that if a $G$-module is divisible by the Steinberg module then it has a $p$-filtration iff it has a good filtration. This in particular applies to the modules in the Steinberg linkage class. We also observe that the equivalence from \cite{A18} between the category of (finite dimensional) $G$-modules and the $n$'th Steinberg component in that category takes modules with a $p^n$-filtration into modules with a $p^{n+1}$-filtration.

In Section 4 we discuss the Donkin conjecture saying that a module has a $p$-filtration if and only if its tensor product with the Steinberg module has a good filtration (i.e. a $p^0$-filtration). We prove for instance that the conjecture holds for all modules  which are divisible by the Steinberg module. Our results in Section 3 also allow us to give a reformulation of the conjecture which is stated purely in terms of modules with $p$-filtrations. More on the Donkin conjecture may be found in \cite{A01},\cite{KN}, \cite{BNPS}.

Finally, in the appendix we take the opportunity to officially withdraw one of the theorems in \cite{A01}. This theorem claimed the main result in the present paper for $p \geq 2(h-1)$,  but unfortunately there is a gap in the proof of a technical lemma on which the theorem was based. The gap was pointed out to me many years ago by S. Donkin (and I withdrew my claimed proof in a lecture at MSRI in 2008). The appendix makes precise exactly which statements in \cite{A01} are affected (fortunately the main results all survive).

The problem treated in Theorem \ref{intro} was first considered by Jantzen in his paper \cite{JCJ80}. He considers the dual case and proves that Weyl modules with ``generic" highest weights have (dual) $p$-filtrations for all $p$. To be ``generic" means to be sufficiently far away from the walls of the dominant chamber. We recover his result in Theorem \ref{small and large} (2), where we have also given precise conditions on the highest weight ensuring it to be ``generic". Then in 2001 my flawed proof appeared. Recently, Parshall and Scott published a paper \cite{PS}, in which they solve the problem for those $p$ which satisfy $p \geq 2h-2$ and for which the Lusztig conjecture on the simple characters for $G$ hold. The latter condition is a serious one: when $G = SL_n$ Williamson \cite{W} has found counter examples to this conjecture for a sequence of $p$'s which grows faster than any polynomial in $n$. Very recently, Bendel, Nakano, Pillen and Sobaje \cite{BNPS2} have found examples of dual Weyl modules for $G$ of type $G_2$ which do not have $2$-filtrations. So our main theorem does not hold in general without restrictions on $p$.

\section{Main result}

\subsection{Conventions and recollections}
For simplicity we shall from now on assume that $R$ is irreducible leaving to the reader the task of generalising to general $R$. We denote the highest short root in $R$ by $\alpha_0$. The Weyl group for $R$ is denoted $W$ and the longest element in $W$ is $w_0$.

In addition to the notation already introduced in the introduction we shall throughout use the following notation (very close to although not completely identical with the conventions in \cite{RAG}).

The $p$-adic components $\lambda^0$ and $\lambda^1$ of a general weight $\lambda \in X$  are defined by the equation
\begin{equation}
\lambda = \lambda^0 + p \lambda^1; \lambda^0 \in X_1, \lambda^1 \in X.
\end{equation}
Note that $\lambda \in X^+$ iff $\lambda^1 \in X^+$.

Recall that $F: G \rightarrow G$ is the Frobenius homomorphism, and its kernel is denoted $G_1$.  We have the corresponding  induction functor

\begin{equation}
 \hat Z_1 = \Ind_B^{G_1B}.
\end{equation}
This is an exact functor. By transitivity of induction we have $\Ind_B^G = \Ind_{G_1B}^G \circ \hat Z_1$. In particular we have $\nabla (\lambda) = \Ind_{G_1B}^G (\hat Z_1(\lambda))$ for all $\lambda$. 

Let $\lambda \in X$. Then $\hat Z_1(\lambda)$ has a unique simple $G_1B$-submodule which we denote $\hat L_1(\lambda)$. Using the above notation we have
\begin{equation}
\hat L_1(\lambda) = L(\lambda^0) \otimes p\lambda^1.
\end{equation}
Here the first factor on the right is the restriction to $G_1B$ of the simple $G$-module $L(\lambda^0)$ and the second factor is the $1$-dimensional $G_1B$-module with trivial $G_1$-action and $B$-action given by $p\lambda^1$.

Note that by the tensor identity and the fact (see e.g. \cite[Proposition I.6.11]{RAG}) that $\Ind_{G_1B}^G \circ -^{(1)} = -^{(1)} \circ \Ind_B^G$ we get from the above
\begin{equation} \label{induced G_1B-simple}
\Ind_{G_1B}^G (\hat L_1(\lambda)) = L(\lambda^0) \otimes \Ind_{G_1B}^G (p\lambda^1) = L(\lambda^0) \otimes \nabla(\lambda^1)^{(1)}.
\end{equation}

The set $\{\hat L_1
(\lambda)_{\lambda \in X}\}$ is up to isomorphisms the set of all finite dimensional simple $G_1B$-modules. The Steinberg module $St = L((p-1)\rho) = \hat L_1((p-1)\rho) = \hat Z_1((p-1)\rho)$ is a special element of this set.

\subsection{Small and large dominant weights}
We begin by showing that if $\lambda \in X^+$ is either ``small" or ``large" with respect to $p$ (see Theorem \ref{small and large} below  for the conventions we use) then $\nabla (\lambda)$ has a $p$-filtration. As mentioned in the introduction the ``large" case was handled by Jantzen, see \cite{JCJ80}. See also \cite {A01}, Lemma 3.4 (for the ``small" case) and Remark 3.7 (for the ``large" case). 

We first  need some weight estimates. They are easy consequences of the well known $T$-structure of the induced module $\hat Z_1(\lambda)$. (Warning: this module is denote $\hat Z_1'(\lambda)$ in \cite{RAG} whereas $\hat Z_1(\lambda)$ there means the coinduced module).  For similar estimates compare \cite{A86}, Section 1.

If $\beta \in R^+$ we set $\h (\beta^\vee) = \langle \rho, \beta^\vee \rangle$ (this is the height of $\beta^\vee$). Note that $\h (\alpha_0^\vee) = h-1$. 

\begin{lem} \label{weight estimates}
Let $\lambda, \mu \in X$. If $\hat L_1(\mu)$ is a $G_1B$-composition factor of $\hat Z_1(\lambda)$ then we have for all $\beta \in R^+$
\begin{equation}
 \langle \lambda^1, \beta^{\vee} \rangle - \h (\beta^\vee) -h +2  \leq  \langle \mu^1, \beta^{\vee} \rangle \leq \langle \lambda^1, \beta^{\vee} \rangle + h - 2.
\end{equation}

\end{lem}

\begin{proof}
As a $T$-module $\hat Z_1(\lambda)$ is isomorphic to $St \otimes (\lambda - (p-1)\rho)$. As $St$ is a $G$-module its weight set is stable under the action of $W$ on $X$. Suppose $\nu$ is a weight of $St$. Then there exists a $w \in W$ for which $w(\nu) \in X^+$. As $w(\nu)$ is a weight of $St$ we have in particular $w(\nu) \leq (p-1)\rho$.  Therefore we get for any $\beta \in R^+$
\begin{equation}\label{weights St}
|\langle \nu, \beta^\vee \rangle | = |\langle w(\nu), w(\beta^\vee) \rangle | \leq \langle w(\nu), \alpha_0^\vee \rangle 
 \leq \langle (p-1)\rho, \alpha_0^\vee \rangle = (p-1)(h-1).
\end{equation}

Let now $\hat L_1(\mu)$ be a composition factor of $\hat Z_1(\lambda)$. Then $\mu$ and $\mu' = w_0(\mu^0) + p \mu^1$ are certainly weights of $\hat Z_1(\lambda)$, and hence by the above we can write $\mu = \lambda - (p-1) \rho + \nu$  and  $\mu' = \lambda - (p-1) \rho + \nu'$ for some weights $\nu, \nu'$ of $St$. Using (\refeq{weights St}) we then get for any $\beta \in R^+$
$$ p\langle \mu^1, \beta^\vee \rangle \leq \langle \mu, \beta^\vee \rangle = \langle \lambda, \beta^\vee \rangle -(p-1) \langle \rho, \beta^\vee \rangle + \langle \nu, \beta^\vee \rangle = $$
$$p \langle \lambda^1, \beta^\vee \rangle + \langle\lambda^0 - (p-1) \rho, \beta^\vee \rangle + \langle \nu, \beta^\vee \rangle \leq p \langle \lambda^1, \beta^\vee \rangle + (p-1)(h-1).$$ This gives the second inequality in the lemma.

Arguing in a similar manner we obtain 
$$ p\langle \mu^1, \beta^\vee \rangle \geq \langle \mu', \beta^\vee \rangle = \langle \lambda, \beta^\vee \rangle -(p-1) \langle \rho, \beta^\vee \rangle + \langle \nu', \beta^\vee \rangle \geq $$
$$ p \langle \lambda^1, \beta^\vee \rangle -(p-1) \h (\beta^\vee) - (p-1)(h-1). $$
This proves the first inequality.

\end{proof}

Now let  $\lambda \in X^+$. Consider a $G_1B$-composition series of $\hat Z_1(\lambda)$
$$0 = F_0 \subset F_1 \subset \cdots \subset F_r = \hat Z_1(\lambda).$$
Then $F_j/F_{j-1} = \hat L_1(\mu_j)$ for some $\mu_j \in X,\;  j= 1, 2, \cdots , r$. Applying the induction functor $\Ind_{G_1B}^G$ to this composition series gives a filtration 
$$0 = F'_0 \subset F'_1 \subset \cdots \subset F'_r =  \Ind_{G_1B}^G (\hat Z_1(\lambda)) = \nabla(\lambda)$$
with $F_j' = \Ind_{G_1B}^G (F_j).$

For each $j$ we have an exact sequence
\begin{equation} \label{short seq}
0 \to F'_{j-1} \to F'_j \to L(\mu_j^0) \otimes \nabla(\mu_j^1)^{(1)}
\end{equation}
where we have identified the last term via (\refeq{induced G_1B-simple}). Note in particular that $F'_j = F'_{j-1}$ whenever $\lambda_j \notin X^+$.

We collect this in the following lemma

\begin{lem} \label{weak filtration}
Let $\lambda \in X^+$. Then $\nabla (\lambda)$ has a filtration with quotients being submodules of $L(\mu_j^0) \otimes \nabla (\mu_j^1)^{(1)}$ for some $\mu_j \in X^+$. 
\end{lem}

With additional assumptions on $\lambda$ we can improve on this result. First we record the following proposition handling ``small", respectively ``large" dominant weights. 

\begin{thm} \label{small and large}
Let  $\lambda \in X^+$. Then
\begin{enumerate}
\item (``small" dominant weights) Suppose $\langle \lambda^1, \alpha_0^\vee
 \rangle \leq p-2h+3$. Then $\nabla (\lambda)$ has a $p$-filtration.
\item (``large" dominant weights) Suppose $\lambda \in p(h-2)\rho + X^+$. Then  $\nabla (\lambda)$ has a $p$-filtration.
\end{enumerate}
\end{thm}

\begin{proof} 
The condition in (1) ensures by Lemma \ref{weight estimates} that all dominant $\mu_j$'s occurring in Lemma \ref{weak filtration} have $\mu_j^1$ in the bottom $p$-alcove in $X^+$, i.e. $\langle \mu_j^1 + \rho, \alpha_0^\vee \rangle \leq p$. This implies by the strong linkage principle \cite{A80} that
$\nabla (\mu_j^1)$ is simple. Hence also  $L(\mu_j^0) \otimes \nabla(\mu_j^1)^{(1)}$ is simple and (1) follows from Lemma \ref{weak filtration}.

To prove (2) we use the fact that the sequence (\refeq{short seq}) is the first part of a long exact sequence
\begin{equation}
0 \to F'_{j-1} \to F'_j \to L(\mu_j^0) \otimes \nabla (\mu_j^1)^{(1)}  \to R^1\Ind_{G_1B}^G (F_{j-1}) \to R^1\Ind_{G_1B}^G (F_{j}) \to R^1\Ind_{G_1B}^G (\hat L_1(\mu_j)) \to \cdots 
\end{equation}
Now as in the proof of Lemma \ref{weak filtration} we get $ R^1\Ind_{G_1B}^G (\hat L_1(\mu_j)) = L(\mu_j^0) \otimes R^1\Ind_B^G(\mu_j^1)^{(1)}$. Our assumption on $\lambda$ ensures (by Lemma \ref{weight estimates}) that all $\mu_j$ have $\langle \mu_j^1, \alpha^\vee \rangle \geq \langle \lambda^1, \alpha^\vee \rangle - h + 1 \geq -1$ for all $\alpha \in S$. Hence by Kempf's vanishing theorem $R^1\Ind_B^G(\mu_j^1) = 0$ for all $j$. By induction on $j$ this means that 
$R^1\Ind_B^G(F_j) = 0$ for all $j$ and hence all sequences in (\refeq{short seq}) are short exact (possibly with last term equal to $0$, namely if there is a simple root $\alpha$ with $\langle \mu_j^1,  \alpha^\vee \rangle = -1$).

\end{proof}

\begin{examplecounter} \label{SL3}
Suppose $G = SL_3$. Then $h = 3$. Let $\lambda $ be a dominant weight and write $\lambda = (a,b)$ to mean $\lambda = a\omega_1 + b\omega_2$ where $\omega_1$ and $\omega_2$ are the two fundamental weights. Write $a = a^0 + p a^1, b = b^0 + p b^1$ with $0 \leq a^0, b^0 < p$. In this notation Theorem \ref{small and large} says that $\nabla (\lambda)$ has a $p$-filtration provided that either $a^1 + b^1 \leq p-3$ or $a, b \geq p$. 

Note that the proposition does not give $p$-filtrations for $\nabla(\lambda)$ for the following (infinite!) set of dominant weights:
$$ \{(a,b) | a \geq p(p-2), 0 \leq b \leq p-1 \} \cup \{(a,b) | 0\leq a \leq p-1, b \geq p(p-2)\}.$$
Using the detailed knowledge of the $G_1B$ composition factors of $\hat Z_1(\lambda)$ in this case it is easy to check via the methods in this section that $p$-filtrations actually exist for all $\lambda \in X^+$ and all $p$. Later we shall improve our results which in the case at hand will also take care of all $p$ as we shall demonstrate in Example 2 (2) below.  Alternatively, see 3.13 in \cite{JCJ80} or \cite{AP}.
\end{examplecounter}

\subsection{The general case}

Let $\lambda \in X^+$ and consider as in the previous subsection a $G_1B$-composition series for $\hat Z_1(\lambda)$. The submodules in this series are again denoted $F_j, \; j= 1, 2, \cdots r$, and the sections are $F_j/F_{j-1} = \hat L_1(\mu_j) = L(\mu_j^0) \otimes p\mu_j^1$. 

Set 
$$ I_\lambda = \{\alpha \in S | \langle \mu_m^1, \alpha^\vee \rangle < -1 \text { for some } m\}.$$
Denote by $R_\lambda = \Z I_\lambda \cap R$ the corresponding root system and by $ P = P_\lambda$ the associated parabolic subgroup. So $P$ is generated by $B$ together with the root subgroups attached to the positive roots in $R_\lambda$.

If $J$ is a connected subset of $S$ we denote by $\alpha_J$ the highest short root of $R_J = \Z J \cap R$ and we set 
$$ h_J = \langle \rho, \alpha_J^\vee \rangle = \h (\alpha_J^\vee).$$
Note that $h_S = h-1$.

Define then
$$ h_\lambda = \max \{h_J | J \text { connected subset of } I_\lambda \} + 1.$$

We shall now consider the induction functor $\Ind_{G_1B}^{G_1P}$. Setting $F_j^{''} = \Ind_{G_1B}^{G_1P}(F_j)$ we argue as we did in establishing (\refeq{short seq}) to see that we have exact sequences

\begin{equation} \label{parabolic sequence}
 0 \to F_{j-1}^{''} \to F_j^{''} \to L(\mu_j^0) \otimes \Ind_B^P(\mu^1_j)^{(1)}.
\end{equation}

\begin{lem} \label{simple P-quotients}
\begin{enumerate}
\item We have $\Ind_B^P (\mu_j^1) \not = 0$ iff $\langle \mu^1_j , \alpha^\vee \rangle \geq 0$ for all $\alpha \in I_\lambda$.
\item Suppose $p \geq (h-2) h_\lambda $. Then $\Ind_B^P (\mu_j^1)$ is a simple $P$-module for each $j$ for which $\langle \mu^1_j , \alpha^\vee \rangle \geq 0$ for all $\alpha \in I_\lambda$. 
\end{enumerate}
\end{lem}

\begin{proof}
The first statement in the lemma is a standard fact about induction from $B$ to $P$. 

To check the second statement consider a connected subset $J \subset I_\lambda$.  We claim that $\langle \mu_j^1 + \rho, \alpha_J^\vee \rangle \leq p$. Let namely $\alpha \in J$ and pick $m$ such that $\langle \mu_m, \alpha^\vee \rangle \leq -2$. The first inequality 
in Lemma \ref{weight estimates} gives 
$$\langle \lambda^1 + \rho, \alpha^\vee \rangle  \leq \langle \mu_m^1 + \rho, \alpha^\vee \rangle + h - 1 \leq h-2.$$
Then by the second inequality in Lemma \ref{weight estimates} we see that 
$$\langle \mu_j^1 + \rho, \alpha_J^\vee \rangle \leq \langle \lambda^1 + \rho, \alpha_J^\vee \rangle + h -2 \leq (h-2) \h (\alpha_J ^\vee) + h-2 \leq (h-2) h_\lambda. $$  
The assumption on $p$ thus ensures that the desired inequality holds. 

Now this being true means that $\mu_j^1$ belong to the bottom dominant alcoves for all connected components of $I_\lambda$. By the strong linkage principle \cite{A80} (applied to the corresponding Levi subgroups) this implies that $\Ind_B^P (\mu_j^1)$ is simple.
\end{proof}

\begin{thm} \label{main}
Let $\lambda \in X^+$. If $p \geq (h-2) h_\lambda$ then  $\nabla (\lambda)$ has a $p$-filtration. 
\end{thm}

\begin{proof}
Let $\lambda \in X^+$. We shall use the notation from above.  According to Lemma \ref{simple P-quotients} the module $\Ind_{G_1B}^{G_1P} \hat Z_1(\lambda)$ has a $G_1P$-filtration $0 = F_0^{''} \subset F_1^{''} \subset \cdots \subset  F_r^{''} = \Ind_{G_1B}^{G_1P} \hat Z_1(\lambda)$ where the quotient $F_j^{''}/ F_{j-1}^{''}$ is either $0$ or a submodule of  $L(\mu_j^0) \otimes \Ind_B^P(\mu_j^1)^{(1)}$. By (2) in Lemma \ref{simple P-quotients} the latter module is a simple $G_1P$-module. Hence we have either $F_j^{''} = F_{j-1}^{''}$ or  a short exact sequence
$$ 0 \to F_{j-1}^{''} \to F_j^{''} \to L(\mu_j^0) \otimes \Ind_B^P(\mu_j^1)^{(1)} \to 0.$$
We now apply the functor $\Ind_{G_1P}^G$ to this filtration. Note that this functor applied to $\Ind_{G_1B}^{G_1P} \hat Z_1(\lambda)$ gives $\nabla (\lambda)$ because by transitivity of induction the composite $\Ind_{G_1P}^G \circ \Ind_{G_1B}^{G_1P} \circ \Ind_B^{G_1B}$ equals  $\Ind_B^G$. Moreover, $\Ind_{G_1P}^G(L(\mu_j^0) \otimes \Ind_B^P(\mu_j^1)^{(1)}) = L(\mu_j^0) \otimes \nabla(\mu_j^1)^{(1)}$. Finally, observe that since $R^1\Ind_P^G (\Ind_B^P(\mu_j)) = R^1 \Ind_B^G(\mu_j^1) = 0$ for all $j$'s for which the quotient $F_j^{''}/F_{j-1}^{''}$ is non-zero (because for such $j$ we have $\mu_j^1 + \rho $ is dominant) the resulting sequence of $G$-submodules is a $p$-filtration of $\nabla (\lambda)$. 
\end{proof}

Noting that $h_\lambda \leq h$ for all $\lambda$ (with equality only if $I_\lambda = S$) we obtain from this theorem the result stated in Theorem \ref{intro} in the introduction, namely

\begin{cor} \label{main cor}
Suppose $p \geq (h-2)h$. Then $\nabla(\lambda)$ has a $p$-filtration for all $\lambda \in X^+$.
\end{cor}

\begin{rem} \label {1 wall}
Suppose $\lambda$ is a dominant weight for which all connected components of $I_\lambda$ consist of just one element. By the definition of $h_\lambda$ this is equivalent to  $h_\lambda = 2$. 
Hence Theorem \ref{main} says that $\nabla(\lambda)$ has a $p$-filtration for all $p \geq (h-2)2$. 

In particular, if $\lambda$ is close to just one wall of the dominant chamber, or more precisely if $I_\lambda = \{\alpha\}$ for some $\alpha \in S$, then $ \nabla(\lambda)$ has a $p$-filtration for all such $p$. Note that by Lemma \ref{weight estimates} we have $I_\lambda \subset \{\alpha\}$ if $\langle \lambda^1, \beta^\vee \rangle \geq h-2$ for all $\beta \in S\setminus \{\alpha\}$. So setting 
$$ X(\leq 1) = \{\lambda \in X^+ | \text {there exists at most one $\alpha \in S$ with } \langle \lambda^1, \alpha^\vee \rangle < h-2\}.$$
we get
\begin{equation} \label{1 wall eq} 
\nabla (\lambda) \text { has a $p$-filtration for all } \lambda \in X(\leq 1) \text { whenever } p \geq 2(h-2).
\end{equation}
\end{rem}

\begin{examplecounter} \begin{enumerate}
\item
Suppose $G$ has rank $2$. Then $X^+\setminus X(\leq 1)$ is contained in the {\it finite} set 
$$Y(\leq 1) =\{\lambda \in X^+ | \langle \lambda^1, \alpha^\vee \rangle < h-2, \alpha \in S\}.$$ 
So for all dominant weights except possibly finitely many the dual Weyl modules for $G$ all have $p$-filtrations for $p \geq 2(h-2)$.

\item
Let us return to the group $G =SL_3$ considered in Example \ref{SL3}. In this case $h = 3$ so by Corollary \ref{main cor} all dual Weyl modules for $SL_3$ have $p$-filtrations when $p \geq 3$. 

Note that for $SL_3$ the finite set $Y(\leq 1)$ consists of just the set of restricted weights. So by the observation in (1) above we get that $\nabla (\lambda)$ has a $p$-filtration for all $p$ except possibly for $\lambda \in X_1, p=2$. However, it is easy to check that for $p=2$ the dual Weyl modules corresponding to the $4$  restricted weights are all simple and hence trivially have a $2$-filtration.

So we have reproved (the known result mentioned in Example \ref{SL3} saying) that all dual Weyl modules for $SL_3$ have $p$-filtrations for all $p$.

\item Let now $G = Sp_4$ . The corresponding root system is $B_2$, which has $1$ short simple root $\alpha_1$ and $1$ long simple root $\alpha_2$. It has $h = 4$ and Corollary \ref{main cor} thus gives $p$-filtrations for all $\nabla(\lambda)$ when $p > 7$. If we limit ourselves to $\lambda \in X(\leq 1)$ we can improve this to $p \geq 5$ by applying instead (\refeq{1 wall eq}).

The finite set $Y(\leq 1)$ is in this case equal to $\{\lambda \in X^+ | \langle \lambda, \alpha_i^\vee \rangle < 2p, i= 1, 2\}$. 
When $p = 7$ direct inspection shows that if $L(\mu)$ is a composition factor of some $\nabla(\lambda)$ with $\lambda \in Y(\leq 1)$ then $\mu^1$ belongs to the lowest alcove in $X^+$. Hence all dual Weyl modules for $Sp_2$ have $7$-filtrations.

The same argument does not work for $p=5$. In this case our methods above give $5$-filtrations for all dominant weights except the $25$ weights belonging to $5 \rho + X_1$. We can handle each $\lambda \in 5\rho + X_1$ by a careful inspection of the composition factors of $\hat Z_1(\lambda)$. Suppose $\lambda$ is $5$-regular, i.e. belongs to the interior of an alcove. Then $\hat Z_1(\lambda)$ has $20$-composition factors (this is true for all $p \geq 5$ and all $p$-regular dominant $\lambda$, see \cite{JCJ77}). If $\lambda$ belongs to one of the top two alcoves in $5\rho + X_1$ a close inspection of the patterns for type $B_2$ on  p. 456 in \cite{JCJ77} reveals, that all composition factors $\hat L_1(\mu)$ of $\hat Z_1(\lambda)$ have $\mu^1 + \rho \in X^+$. In this case the arguments used in the proof of Theorem \ref{small and large} (2) produce a $5$-filtration for $\nabla(\lambda)$. If $\lambda$ belongs to one of the two lower alcoves in $5\rho + X_1$ then there is exactly $1$ composition factor $\hat L_1(\mu)$ of $\hat Z_1(\lambda)$ with $\mu^1 + \rho \notin X^+$. Let $\mu_+ = \mu^0 + 5 \mu_+^1$ with $\{\mu_+^1\} = W \cdot \mu^1 \cap X^+$. Then by inspection we observe that $L(\mu_+)$ is not a composition factor of $\nabla (\lambda)$. The arguments in Section 2 then show that $\Ind_B^G$ also in this case takes a composition series of $\hat Z_1(\lambda)$ into a $5$-filtration of $\nabla (\lambda)$. Finally, if $\lambda$ is not $5$-regular  $\hat Z_1(\lambda)$ has much fewer composition factors (at most $10$) and the same arguments work.  Hence in fact  all dual Weyl modules for $Sp_4$ have $5$-filtrations.

When $p$ is either $2$ or $3$ our best result is Proposition \ref{small and large} (2), which however leaves us with infinitely many weights $\lambda$ for which the question of  whether $\nabla(\lambda) $ has a $p$-filtration is open. A tedious check of the composition patterns of each of the corresponding $\hat Z_1(\lambda)$ reveals that $\nabla(\lambda)$ does have a $p$-filtration in all these cases. When $p = 2$ an extra subtlety occurs: Some of the $\hat Z_1(\lambda)$'s have composition factors with multiplicities $> 1$. For instance, $\hat Z_1(0)$ has the following $8$ composition factors (we give their highest weights in terms of their coordinates with respect to the fundamental weights) 
$$\hat L_1(0,0), \hat L_1(-2,1), \hat L_1(2,-2), \hat L_1(0,-1), \hat L_1(-2,0), \hat L_1(0,-2), \hat L_1(-4,0),\hat L_1(-2,2).$$ Of these $\hat L_1(-2,0)$ and $\hat L_1(0,-2)$ occur with multiplicity $2$. Nevertheless one checks that $\Ind_{G_1B}^G$ still takes a composition series for $\hat Z_1(\lambda)$ into a $2$-filtration for $\nabla (\lambda)$ for all $ \lambda$.
\end{enumerate}
\end{examplecounter}

\section{$p^n$-filtrations}

Let $n \geq 0$ and define $X_n$ to be
$$ X_n = \{\lambda \in X^+ | \langle \lambda, \alpha^\vee \rangle < p^n \text { for all } \alpha \in S\}.$$
The elements of $X_n$ are called the $p^n$-restricted weights. We now write for any $\mu \in X$
$$ \mu = \mu^0 + p^n \mu^1$$
with $\mu^0 \in X_n$ and $\mu^1 \in X$. Here we are in conflict with the notation used in previous sections but we will make sure not to mix it up with the $n=1$ notation considered so far.

If $M$ is a $G$-module we denote by $M^{(n)}$ the twist by $F^n$ of $M$. We have the inductive formula $M^{(n)} = (M^{(n-1)} )^{(1)}$. Note that $(M^{(n)})^{(m)} = M^{(n+m)}$ for all $m \geq 0$. 

\subsection{Higher filtrations}

We say in analogy with the $n=1$ case that a $G$-module $M$ has a $p^n$-filtration if it has a series of $G$-submodules
$$ 0 = F_0 \subset F_1 \subset \cdots \subset F_r = M$$
with $F_j/F_{j-1} = L(\mu_j^0) \otimes \nabla(\mu_j^1)^{(n)}$ for some $\mu_j \in X^+$. 

Note that a $1$-filtration is the same as a good filtration.

\begin{prop}\label{p^n-filt}
Suppose $p \geq h(h-2)$. If $M$ has a $p^n$-filtration then $M$ has also a $p^m$-filtration for all $m \geq n$. 
\end{prop}

\begin{proof}
It is clearly enough to check the proposition in the case where $m = n+1$ and $M = L(\lambda^0) \otimes \nabla (\lambda^1)^{(n)}$ for some $\lambda \in X^+$. Now by Corollary \ref{main cor} the dual Weyl module $\nabla(\lambda^1)$ has a $p$-filtration, i.e. a filtration with quotients of the form  $L(\mu) \otimes \nabla (\nu)^{(1)}$ with $\mu \in X_1$ and $\nu \in X^+$. Then $M$ has a filtration with quotients $L(\mu^0) \otimes L(\mu)^{(n)} \otimes \nabla(\nu)^{(n+1)}$. By Steinberg's tensor product theorem $L(\mu^0) \otimes L(\mu)^{(n)} = L(\mu^0 + p^n \mu)$ and since $\mu^0 + p^n \mu \in X_{n+1}$ we have thus obtained a $p^{n+1}$-filtration of $M$.
\end{proof}

 \begin{rem}
\begin{enumerate}
\item The case $n= 0, \, m= 1$ in this proposition is equivalent to our main result, Theorem \ref{intro}.
\item
Note that if $m > 0$ not all modules with $p^{n+m}$-filtrations have $p^n$-filtrations. Examples are for instance all $L(\lambda)$ where $\lambda \in X_{n+m} \setminus X_n$ with $L(\lambda) \neq \nabla(\lambda)$.
\end{enumerate}
\end{rem}

\begin{cor} \label{main cor-n}
If $p \geq (h-2)h$ then all dual Weyl modules have $p^n$-filtrations for all $n \geq 0$.
\end{cor}

\begin{proof}
Immediate from Corollary \ref{main cor} and Proposition \ref{p^n-filt}.
\end{proof}

\subsection{Tensor products involving Steinberg modules}

Let $n \geq 0$. The $n$'th Steinberg module is 
$$ St_n = L((p^n -1)\rho).$$
Note that $St_1 = St$ in our notation in Section 2. By the Steinberg tensor product theorem we have $St_n = St_{n-1} \otimes St^{(n-1)} = St \otimes St_{n-1}^{(1)}$. By the strong linkage principle we have  $St_n = \nabla((p^n-1)\rho)$.

As $St_n$ is self dual we have
\begin{equation} \label{Steinberg summand}
St_n \text { is a $G$-summand of } St_n^{\otimes 3}.
\end{equation}

\begin{defn}
Let $M$ be a $G$-module. We say that $M$ is divisible by $St_n$ if there exists a $G$-module $N$ such that $M = St_n \otimes N$.
\end{defn}

Note that $St_n$ is divisible by $St_m$ for all $m \leq n$.
 
Let $M$ be a $G$-module. Then 

\begin{prop} \label{tensor2}  Suppose $M$ is divisible by $St_n$. Then $M$ has a good filtration iff $M\otimes St_n$ has a good filtration.
\end{prop}

\begin{proof}
The forward implication holds always due to the Wang-Donkin-Mathieu theorem, \cite{Wa}, \cite{Do}, \cite{Ma}, on tensor products of dual Weyl modules. To check the converse write $M = St_n \otimes N$. Applying this theorem again we see that  if $M \otimes St_n$ has a good filtration  so does $M \otimes St_n^{\otimes 2} = N  \otimes St_n^{\otimes 3}$. But by (\refeq{Steinberg summand}) we see that  this latter module has $M$ as a summand and hence  $M$ has a good filtration.
\end{proof}

\begin{prop} \label{n to good}
Suppose $p \geq (h-2)h$ and let $M$ be a $G$-module which is divisible by $St_n$.  Then $M$ has a $p^n$-filtration iff $M$ has a good filtration.
\end{prop}

\begin{proof}
If $M$ has a good filtration we get from Corollary \ref{main cor-n}, that $M$ also has a $p^n$-filtration. Conversely, if $M$ has a $p^n$-filtration then by \cite{A01} Proposition 2.10 we get that $M \otimes St_n$ has a good-filtration for $p \geq 2h-2$. Then Proposition \ref{tensor2} says that $M$ has a good filtration as well.

Note than in the above argument we need $p \geq 2h-2$. This is implied by our assumption $p \geq (h-2)h$ unless $ p = h = 3 $.  But then $G$ is $SL_3$ and the result is easily checked by a direct computation. Alternatively, see \cite{BNPS} where the bound $2h-2$ is improved to $2h-4$.
\end{proof}

\begin{rem} \label{sharper}

We could (with the same assumption on $p$) sharpen the proposition to the (seemingly) more general statement about a $G$-module which is divisible by $St_n$.
$$ \text{Let $m \leq n$. Then $M$ has a $p^m$-filtration iff $M$ has a good filtration}. $$
This follows from the fact that $St_m$ is a tensor factor in $St_n$ for all $m \leq n$. 

\end{rem}

\begin{cor} Suppose $p \geq (h-2)h$ and let $M$ be a $G$-module which is divisible by $St_n$. Then 
$M$ has a $p^n$-filtration iff $M \otimes St_n$ has a $p^n$-filtration.
\end{cor}

\begin{proof}
Combine Propositions \ref{tensor2} and \ref{n to good}.
\end{proof}

\subsection{Relations to the Steinberg component}

Recall that in \cite{A18} we establish an equivalence between the category of rational $G$ -modules and the Steinberg component of this category. The $n$'th Steinberg component consists of all $G$-modules whose composition factors have the form $L(p^n \cdot \lambda)$ with $\lambda \in X^+$ (using the convention $p^n \cdot \lambda = p^n(\lambda + \rho) - \rho)$. The equivalence is the composite of twisting with the Frobenius $n$ times and tensoring with the $n$'th Steinberg module. We shall now see how this equivalence behaves with respect to $p$-filtrations and their higher analogues.

\begin{prop}\label {m,n} Let $M$ be a $G$ module and $m, n \in \Z_{\geq 0}$. Then $M$ has a $p^m$-filtration iff $M^{(n)} \otimes St_n$ has a $p^{(m+n)}$-filtration. 
\end{prop}

\begin{proof}
Let $\lambda \in X^+$. In analogy with Section 3.1 we  write $\lambda = \lambda^0 + p^m \lambda^1$ with $\lambda \in X_m$. Then we have $(L(\lambda^0) \otimes \nabla(\lambda^1)^{(m)})^{(n)}  \otimes St_n = St_n\otimes  L(\lambda^0)^{(n)} \otimes \nabla(\lambda^1)^{(n+m)} = L((p^n-1)\rho + p^n \lambda^0) \otimes \nabla(\lambda^1)^{(n+m)}$ where we have used once more the Steinberg tensor product theorem. This proves the only if statement. 

To check the converse let $N$ be a $G$-module belonging to the $n$'th Steinberg component. Then $N = St_n \otimes M^{(n)}$ with $M = \Hom_{G_n} (St_n, N)^{(-n)}$, see \cite{A18}. Applying the exact functor $\Hom_{G_n}(St_n, -)^{(-n)}$ to a $p^{(m+n)}$-filtration of $N$ will give us the desired $p^n$-filtration of $M$. In fact, the value of this functor on a module like  $St_n \otimes (L(\mu)^{(n)} \otimes  \nabla(\nu)^{(m+n)}$, where $\mu \in X_m$ and $\nu \in X^+$, is $ 
L(\mu) \otimes \nabla(\nu)^{(m)}$.

\end{proof}

\section{Donkin's conjecture on $p$-filtrations}

In 1990 S. Donkin proposed the following conjecture. We abbreviate it $DC_1$.

Let $M$ be a $G$-module. Then 
\begin{equation} \label{DC_1}
  \text {$ M$ has a $p$-filtration iff $M \otimes St$ has a good filtration.}
\end{equation}
There is an obvious higher version of $DC_1$ which we name $DC_n$, namely
\begin{equation} 
 \text {$M$ has a $p^n$-filtration iff $M \otimes St_n$ has a good filtration.}
\end{equation}

In this section we shall make some remarks on these conjectures. In particular, we shall relate them to the results in the previous sections. 

\subsection{What is known, partially known, or unknown about $DC_n$}
\begin{rem}
\begin{enumerate}
\item The implication from left to right $DC_1$ was proved in \cite{A01} for $p \geq 2h-2$. See also \cite{KN} for an alternative proof and \cite{BNPS} for a lowering of the bound on $p$ to $p \geq 2h-4$. The converse implication is only known to hold for $SL_2$, see \cite{A01} Proposition 4.3. The conjecture is also open in both directions when $p$ is small.

\item As we have observed before $\nabla(\lambda) \otimes St$ has a good filtration for all $\lambda$ because of the Wang-Donkin-Mathieu theorem, cf. \cite{Wa}, \cite{Do}, \cite{Ma}. Hence if the right to left implication in $DC_1$ is proved, our main result Theorem \ref{intro} would be a consequence (for all $p$). We consider our result as partial evidence for the conjecture.

\item It is rather easy to see that if the left to right implication holds in $DC_1$ then the same implication holds in $DC_n$, see Proposition 2.10 in \cite{A01}.  In contrast this author knows of no ways to reduce the reverse implications to the case $n=1$.
\end{enumerate}
\end{rem}

\begin{prop} If $DC_{n+1}$ holds then so does $DC_n$.
\end{prop}

\begin{proof}
Assume $DC_{n+1}$ and let $M$ be a $G$-module. Observe first that $M$ has a $p^n$-filtration iff $M^{(1)}$ has a $p^{n+1}$-filtration. By $DC_{n+1}$ this is the case iff $M^{(1)} \otimes St_{n+1} = (M \otimes St_n)^{(1)} \otimes St_1$ has a good filtration, i.e. iff $M \otimes St_n$ has a good filtration. Here the last step follows from \cite{A18} Corollary 3.2 (3)  applied to $M \otimes St_n$.
\end{proof}

The Propositions \ref{tensor2} and \ref{n to good} together with Remark \ref{sharper} give us a big family of modules for which $DC_n$ is true.

\begin{prop} Suppose $p \geq (h-2)h$. If $N$ is a $G$-module which is divisible by $St_n$ then $DC_n$ holds for $N$. In particular, $DC_n$ holds (as do in fact $DC_m$ for all $m \leq n$) for all modules belonging to the $n$'th Steinberg component.
 \end{prop}

Proposition \ref{n to good} also allow us to reformulate the Donkin conjectures.

\begin{prop}
When $p \geq (h-2)h$, $DC_n$ is equivalent to the following statement: Let $M$ be a $G$-module. Then 
$$ M \text { has a $p^n$-filtration iff $M \otimes St_n$ has one. }$$
\end{prop}

\subsection{$\Ext$-criteria}

Recall that if $\lambda \in X^+$ then there is a unique indecomposable tilting module $T(\lambda)$ with highest weight $\lambda$. In Theorem 2.4 in \cite{A01} we proved the following criteria for $M \otimes St_n$ to have a good filtration (actually this theorem was only proved for $n=1$ in {\it loc.cit.} but the argument for arbitrary $n$ is the same)
$$ \text { $ M \otimes St_n$ has a good filtration iff $\Ext^j_G(T(\lambda), M) = 0$ for all $j > 0$ and all $\lambda \in (p^n-1)\rho + X^+$.}$$

Note that the sets $(p^n-1)\rho + X^+$ decrease (strictly) with $n$. This means that if $M \otimes St_n$ has a good filtration 
for some $n$ then $M \otimes St_m$ has a good filtration also for all $m \geq n$. This is consistent with Proposition \ref{p^n-filt}.

\section{Appendix}

In \cite{A01} Corollary 3.7  I claimed to prove that all dual Weyl modules have $p$-filtrations for $p \geq 2h-2$. However, S. Donkin has pointed out to me that there is a problem with my proof of a lemma which is crucial for this result. I therefore withdrew my claim in a lecture at MSRI in 2008, see \cite{A08}. In this appendix I use the opportunity to record this retraction in writing \footnote{The counterexamples in \cite{BNPS2} shows that the lemma is in fact false (at least for $p=2$)} and to point out exactly where the problem is.

I'm grateful to S. Donkin for pointing out the gap in my proof, and to P. Sobaje for helpful comments. 

\vskip .5 cm
Lemma 3.3 in {\it loc. cit.} states that if a $G$-module $M$ has a submodule $M' \subset M$ such that both $M$ and $M'$ have $p$-filtrations then so does the quotient $M/M'$. Note that this lemma is true whenever Donkin's conjecture from Section 4 is true. Likewise the lemma also holds if there is an $\Ext$-vanishing criteria for a module to have a $p$-filtration.

However, my proof of Lemma 3.3 in \cite{A01} contains a gap: We claim (without proof) that if a $G$-module $M$ has two submodules $ M_1 = L(\lambda) \otimes \nabla(\mu)^{(1)}$ and $M'_1 = L(\lambda) \otimes \nabla(\nu)^{(1)}$ with $\lambda \in X_1$ and $\mu, \nu \in X^+$  then either $M_1 = M'_1$ or $M_1 \cap M_1' = 0$. This is not true: Take e.g. $G = SL_2$, $M = \nabla(p)^{(1)} \oplus \nabla(p-2)^{(1)}, M_1 = \nabla(p)^{(1)}$ and $M_1' = \{(x,f(x)) | x \in M_1 \}$ with $f: \nabla(p)^{(1)} \to \nabla(p-2)^{(1)}$ being non-zero. In this situation $M_1 \cap M'_1 = L(p)^{(1)}$. 

In addition to Lemma 3.3 I also withdraw Theorem 3.6 and Corollary 3.7, which both rely in an essential way on this lemma. Fortunately, all remaining results in the paper are independent of these results. 

The present paper contains a completely different proof of the results in Theorem 3.6 and Corollary 3.7 (under a stronger assumption on $p$).

\vskip 1 cm


\vskip 1 cm 
 \begin{thebibliography}{}

\bibitem{A80} H.H.~Andersen, The strong linkage principle, J. Reine Ang. Math. 315 (1980), 53-59.
\bibitem{A86} H.H.~Andersen, On the generic structure of cohomology modules for semisimple algebraic groups . Trans. Amer. Math. Soc. 295 (1986) 397-415.
\bibitem{A01} H.~H.~Andersen, $p$-filtrations and the Steinberg module, J. Algebra 244 (2001), 664-683.
\bibitem{A08}  H.~H.~Andersen, Tensoring with the Steinberg module, MSRI lecture, April 2, 2008.
http://www.msri.org/web/msri/online-videos/-/video/showSemester/2008010120080701
\bibitem{A18} H.~H.~Andersen, The Steinberg linkage class for a reductive algebraic group, Arkiv f$\o$r Matematik 56 (2018), 229-241.
 
\bibitem{BNPS} C.P. Bendel, D.K.~Nakano, C.~Pillen, and P.~Sobaje, On tensoring with the Steinberg representation, Transformation Groups (to appear), online available: ArXiv: 1804.00613.
\bibitem{BNPS2} C.P. Bendel, D.K.~Nakano, C.~Pillen, and P.~Sobaje, Counterexamples to the Tilting and $(p,r)$-Filtration Conjectures, ArXiv:1901.06687.
\bibitem{Do} S. Donkin, Rational Representations of Algebraic Groups, Lecture Notes in Math. 1140 (Springer 1985).
\bibitem{JCJ77} J.C.~Jantzen, $U$ber das Dekompositionsverhalten gewisser modularer Darstellungen halbeinfacher Gruppen und ihrer Lie-Algebren. J. Algebra 49 (1977), 441-469.
\bibitem{JCJ80} J.C.~Jantzen, Darstellungen halbeinfacher Gruppen und ihrer Frobenius-Kerne, J. reine angew. Math.
317 (1980), 157-199.
\bibitem{RAG} J.C.~Jantzen, {\em Representations of Algebraic Groups}, Mathematical Surveys and Monographs 107, Second edition, American Mathematical Society (2003).
\bibitem{KN} T.~Kildetoft and D.K.~Nakano, On good $(p,r)$-filtrations for rational G-modules, J. Algebra 423 (2015), 702-725.
\bibitem{AP} A. Parker, Good $l$-filtrations for $q$-$\rm GL_3(k)$, J. Algebra 304 (2006), 157-189.
\bibitem{PS} B. Parshall and L. Scott, On $p$-filtrations of Weyl modules, J. London Math. Soc. 91 (2015), 127-158.

\bibitem{Ma} O. Mathieu, Filtrations of G-modules, Ann. scient. Ec. Norm. Sup. 23 (1990), 625-644.
\bibitem{Wa} J.-p.~Wang, Sheaf cohomology on G/B and tensor products of Weyl modules, J. Algebra 77 (1982), 162-185.
\bibitem{W} G.~Williamson, Schubert calculus and torsion explosion, J. Amer. Math. Soc. 30 (2017), 1023-1046.
\end{thebibliography}
\end{document}